\documentclass[10pt]{amsart}%
\usepackage{eurosym}
\usepackage[mathscr]{eucal}
\usepackage{graphicx,amscd,color,amsmath,amsfonts,amssymb}
\usepackage{amsmath}
\usepackage{amsfonts}
\usepackage{amssymb}
\usepackage{graphicx}%
\providecommand{\U}[1]{\protect\rule{.1in}{.1in}}

\newtheorem{theorem}{Theorem}
\newtheorem{proposition}[theorem]{Proposition}
\newtheorem{corollary}[theorem]{Corollary}
\newtheorem{example}[theorem]{Example}

\newtheorem{lemma}[theorem]{Lemma}

\numberwithin{equation}{section}

\begin{document}
	\title[Geometry of multilinear forms]{On the geometry of multilinear forms}
	\author{W.V. Cavalcante} 
	\author{D.M. Pellegrino}  
	\author{E.V. Teixeira}
	\address{Departamento de Matem\'{a}tica, 
		\indent Universidade Federal de Pernambuco,
		\indent 50.740-560 - Recife, PE, Brazil.}
	\address{
		Departamento de Matem\'{a}tica, 
		\indent Universidade Federal da Para\'{\i}ba, 
		\indent 58.051-900 - Jo\~{a}o Pessoa, Brazil.}
	\address{
		Department of Mathematics, 
		\indent University of Central Florida, 
		\indent Orlando, FL, USA 32816}
	
	\email{wasthenny@dmat.ufpe.br}
	\email{pellegrino@pq.cnpq.br}
	\email{Eduardo.Teixeira@ucf.edu}
	
	\thanks{Authors acknowledge support from Capes and CNPq-Brazil.}
	
	\begin{abstract}
	{We develop a constructive process which determines all extreme points of the
	unit ball of the space of $m$--linear forms, $m\geq1.$ Our method
	provides a full characterization of the geometry of that space through
	finitely many elementary steps, and thus it can be extensively applied in both computational and theoretical problems.}
		
	\end{abstract}
	\maketitle
	\tableofcontents

\section{Introduction}

Mathematical models involving multilinear forms are abundant in applied
sciences, in particular multivariable polynomials represent an endless source
of examples of such matter.  

It is often that intrinsic difficulties in
understanding multilinear problems are manifestations of the geometry complexity of the
space of multilinear forms. As a way of example, we mention the problem of finding sharp constants in
classical multilinear, convex inequalities. Routine applications of the
Krein-Milman Theorem often reduces the candidate set to the extreme points, thus, genuine difficulties in determining sharp
constants heavily rely on the lack of understanding upon the geometry of the
space of multilinear operators.

This is a critical issue resting in the core of pure and applied mathematical
analysis. Previous works on this theme include \cite{cobos, ggg3, ggg};
however up-to-date, only problems involving low dimensions and/or low degrees
have been successfully investigated; see also \cite{bourgain, ggg2, kad} for
related issues. In this article we tackle the problem in full generality.

Let $B_{\mathbb{R}^{n}}$ denote the closed unit ball of $\mathbb{R}^{n}$,
endowed with the $\sup$ norm. We denote the space of all $m$--linear forms
$T\colon\mathbb{R}^{n}\times\cdots\times\mathbb{R}^{n}\mathbb{\rightarrow R}$
by $\mathcal{L}\left(  ^{m}\mathbb{R}^{n}\right)  $. As usual, we equip this
vector space with the norm
\begin{equation}
\left\Vert T\right\Vert :=\sup_{\left\Vert x_{1}\right\Vert ,...,\left\Vert
x_{n}\right\Vert \leq1}|T(x_{1}...,x_{n})|.\label{norma222}%
\end{equation}
The closed unit ball of $\mathcal{L}\left(  ^{m}\mathbb{R}^{n}\right)  $ will
be denoted by $K$, i.e.,
\[
K:=\left\{  T\colon\mathbb{R}^{n}\times\overset{m}{\cdots}\times\mathbb{R}%
^{n}\rightarrow\mathbb{R}:T\text{ is }m\text{-linear and }\left\Vert
T\right\Vert \leq1\right\}  .
\]

The key objective of this paper is to thoroughly characterize the geometry of
$K$, by establishing all of its extreme points, henceforth denoted by
$\mathcal{C}_{m,n}$ or simply by  $\mathcal{C}$. We
describe a procedure involving only finite elementary steps to determine
$\mathcal{C}$. A particularly interesting inference from this process is that
the coordinates of the elements of $\mathcal{C}$ are all rational points. In
the sequel, we investigate optimization problems in classical real
inequalities with the aid of our main characterization theorem. The examples
included here have been nfluenced by the authors' personal taste; however it is
clear that our approach can be applied to a very large class of optimization problems.

The paper is organized as follows. In Section \ref{sct pre} we gather some
preliminary tools and discuss notations to be used throughout the whole
article. In Section \ref{sct geo} we obtain the main results of the paper,
namely Theorem \ref{arr} and Theorem \ref{arrimp}, which determine all extreme
points of the closed unit ball of the space of $m$-linear forms in arbitrary
dimensions. In Section \ref{CPPP} we discuss the algorithm inferred from the
proofs delivered in the previous Section. Applications of the main results in
the investigation of sharp constants in classical real inequalities are
discussed in the last Section 5.

\section{Preliminary results and notations}

\label{sct pre}

As previously commented, throughout the paper, $\mathbb{R}^{n}$ will always be
equipped with the $\sup$ norm, unless mentioned otherwise. Following classical
notations, given a matrix $M$, its transpose is denoted by $M^{t}$. The set
$\{1,\ldots,n\}$ will be denoted by $[n]$. For $x_{1},x_{2},\ldots,x_{m}%
\in\mathbb{R}^{n}$ and $\mathbf{j}=(j_{1},\ldots,j_{m})\in\lbrack n]^{m},$ we
define
\begin{equation}
x^{\mathbf{j}}:=%
{\displaystyle\prod\limits_{i=1}^{m}}
x_{i}^{(j_{i})}\in\mathbb{R}, \label{9d}%
\end{equation}
where $x_{i}^{(j_{i})}$ denotes the $j_{i}$-th coordinate of vector $x_{i}$.
We also define
\begin{equation}
\omega(x):=(x^{\mathbf{j}})_{\mathbf{j}\in\lbrack n]^{m}}\in\mathbb{R}^{n^{m}%
}, \label{9dd}%
\end{equation}
using the lexicographic order.

If $m,n$ are positive integers, let us set
\begin{equation}
V_{m}^{n}:=\left\{  \omega(x):x=\left(  x_{1},x_{2},\ldots,x_{m}\right)
\text{ and }x_{i}\in ext\left(  B_{\mathbb{R}^{n}}\right)  \text{ for all
}i\in\lbrack m]\right\}  . \label{9ddd}%
\end{equation}

Finally, we recall that given a vector space $E$ and a convex set $A\subseteq
E$, a vector $x\in A$ is said to the an extreme point of $A$ if $y,z\in A$
with $x=\left(  y+z\right)  /2$ implies $y=z.$ From now on $ext\left(
A\right)  $ denotes the set of extreme points of $A.$

\subsection{Bases of vertices of hypercubes}

We start off by proving some basic facts about $ext\left(  B_{\mathbb{R}^{n}%
}\right)  $ that will be useful later.

\begin{lemma}
\label{km}(Minkowski/Krein-Milman) If $E$ is a locally convex space and $K$ is
a nonempty convex and compact subset of $E,$ then $K$ has at least one extreme
point and $K=conv(extK)$, where $extK$ is the set of all extreme points of $K$
and $conv(A)$ denotes the closed convex hull of $A$.
\end{lemma}

\begin{lemma}
\label{yyy}There exists a basis of $\mathbb{R}^{n}$ composed by vectors from
$ext\left(  B_{\mathbb{R}^{n}}\right)  $.
\end{lemma}

\begin{proof}
This is a direct consequence of Krein-Milman Theorem.
\end{proof}

\begin{lemma}
\label{tec} Let $v_{1},\ldots,v_{m},u_{1},\ldots,u_{m}\in\mathbb{R}^{n}$.
Then
\[
\langle\omega(v),\omega(u)\rangle=\prod\limits_{i=1}^{m}\langle v_{i}%
,u_{i}\rangle,
\]
where $v=(v_{1},\ldots,v_{m})$ and $u=(u_{1},\ldots,u_{m})$.
\end{lemma}

\begin{proof}
One simply notices that
\begin{align*}
\langle\omega(v),\omega(u)\rangle &  =\sum_{\mathbf{k}\in\lbrack n]^{m}%
}v^{\mathbf{k}}u^{\mathbf{k}}\\
&  =\sum_{\mathbf{k}\in\lbrack n]^{m}}(v_{1}^{(k_{1})}\cdots v_{m}^{(k_{m}%
)})(u_{1}^{(k_{1})}\cdots u_{m}^{(k_{m})})\\
&  =\sum_{\mathbf{k}\in\lbrack n]^{m}}(v_{1}^{(k_{1})}u_{1}^{(k_{1})}%
)\cdots(v_{m}^{(k_{m})}u_{m}^{(k_{m})})\\
&  =\left(  \sum_{k_{1}\in\lbrack n]}v_{1}^{(k_{1})}u_{1}^{(k_{1})}\right)
\left(  \sum_{k_{2}\in\lbrack n]}v_{2}^{(k_{2})}u_{2}^{(k_{2})}\right)
\cdots\left(  \sum_{k_{m}\in\lbrack n]}v_{m}^{(k_{m})}u_{m}^{(k_{m})}\right)
\\
&  =\prod_{i=1}^{m}\langle v_{i},u_{i}\rangle.
\end{align*}

\end{proof}

\begin{proposition}
For all $i\in\lbrack m]$, let
\[
\beta_{i}=\{v_{i,1},\ldots,v_{i,n}\}
\]
be a set of non-null vectors in $\mathbb{R}^{n}$. The following assertions are equivalent:

(i) $\beta_{i}$ is a basis of $\mathbb{R}^{n}$ for all $i\in\lbrack m]$.

(ii) $\Lambda_{m}(\beta_{1},\ldots,\beta_{m}):=\{\omega(x):x=(x_{i})_{i=1}%
^{m}\in\Pi_{i\in\lbrack m]}\beta_{i}\}$ is a basis of $\mathbb{R}^{n^{m}}$.
\end{proposition}

\begin{proof}
$(i)\Rightarrow(ii)$ For all $i\in\lbrack m]$ there is a basis $\gamma
_{i}=\{u_{i,1},\ldots,u_{i,n}\}$ of $\mathbb{R}^{n}$ satisfying
\[
\langle v_{i,r},u_{i,s}\rangle=\delta_{rs}^{i},
\]
where $\delta_{rs}^{i}$ is the Kronecker's delta and $r,s\in\lbrack n]$. Given
$\mathbf{i}=(i_{1},\ldots,i_{m})$ and $\mathbf{j}=(j_{1},\ldots,j_{m}%
)\in\lbrack n]^{m}$, consider
\[
v_{\mathbf{j}}=(v_{1,j_{1}},\ldots,v_{m,j_{m}})
\]
and
\[
u_{\mathbf{j}}=(u_{1,j_{1}},\ldots,u_{m,j_{m}}).
\]
By Lemma \ref{tec}, we have
\[
\delta_{\mathbf{ij}}:=\Pi_{s\in\lbrack m]}\delta_{j_{s}i_{s}}^{s}=\Pi
_{s\in\lbrack m]}\langle v_{s,j_{s}},u_{s,i_{s}}\rangle=\langle\omega
(v_{\mathbf{j}}),\omega(u_{\mathbf{i}})\rangle.
\]

$(ii)\Rightarrow(i)$ Suppose that (i) is not valid. Thus, there is $j_{0}%
\in\lbrack m]$ such that $\beta_{j_{0}}$ is not a basis of $\mathbb{R}^{n}$,
i.e., there is a $k_{0}\in\lbrack n]$ such that
\[
v_{j_{0},k_{0}}=\sum_{i\neq k_{0}}\alpha_{i}v_{j_{0},i}%
\]
for certain scalars $\alpha_{i},$ $i\neq k_{0}.$ Therefore it is immediate
that $\Lambda_{m}(\beta_{1},\ldots,\beta_{m})$ is not composed by linearly
independent vectors.
\end{proof}

\begin{corollary}
\label{constbas2} If $\beta$ is a basis of $\mathbb{R}^{n}$, then
\[
\Lambda_{m}(\beta):=\Lambda_{m}(\beta,\ldots,\beta)
\]
is a basis of $\mathbb{R}^{n^{m}}$.
\end{corollary}

From Lemma \ref{yyy} and Corollary \ref{constbas2} we have the following consequence:

\begin{corollary}
\label{coro1}There exists a basis of $\mathbb{R}^{n^{m}}$ contained in
$V_{m}^{n}.$
\end{corollary}

\subsection{Some algebraic tools}

We denote by $O(n^{m})$ the set of all orthogonal $n^{m}\times n^{m}$
matrices. Given $(c_{i})_{i=1}^{m}\in\mathbb{R}^{m}$, we define $diag(c_{i}%
)_{i\in\lbrack m]}$ to be the $m\times m$ diagonal matrix whose entries are
$c_{i}.$ Let
\[
G_{m}^{n}:=\left\{  diag(x^{\mathbf{j}})_{\mathbf{j}\in\lbrack n]^{m}}\in
O(n^{m}):x=\left(  x_{i}\right)  _{i=1}^{m}\text{ and }x_{i}\in ext\left(
B_{\mathbb{R}^{n}}\right)  \text{ for all }i\in\lbrack m]\right\}  ,
\]
where $x^{\mathbf{j}}$ is as in \eqref{9dd}, and we still use the
lexicographic order.

\begin{proposition}
Let $m,n$ be positive integers. Then

(i) $G_{m}^{n}$ is a subgroup of $O(n^{m})$;

(ii) The map $\phi:G_{m}^{n}\times V_{m}^{n}\rightarrow V_{m}^{n}$ given by
\[
\phi\left(  g,\omega(x)\right)  \mapsto\omega(x)\cdot g,
\]
where $x = (x_{1},\ldots,x_{m})$ and $x_{1},\ldots,x_{m}\in ext(B_{\mathbb{R}%
^{n}})$ is well defined;

(iii) $\phi$ is a free (left) group action.
\end{proposition}

\begin{proof}
(i) The identity belongs to $G_{m}^{n};$ in fact, we just need to consider
$x_{1}=\cdots=x_{m}=(1,1,\ldots,1)\in ext\left(  B_{\mathbb{R}^{n}}\right)  $.
If $g\in G_{m}^{n}$, then
\[
g^{-1}=g^{t}=g\in G_{m}^{n}.
\]
Now, let us show that $G_{m}^{n}$ is closed under multiplication. Given
$g,h\in G_{m}^{n}$ there are $x_{1},\ldots,x_{m},y_{1},\ldots,y_{m}\in
ext\left(  B_{\mathbb{R}^{n}}\right)  $ such that
\[
g=diag(x^{\mathbf{j}})_{\mathbf{j}\in\lbrack n]^{m}}\text{ and }%
h=diag(y^{\mathbf{j}})_{\mathbf{j}\in\lbrack n]^{m}}.
\]
Define $z_{i}=(x_{i}^{(s)}y_{i}^{(s)})_{s\in\lbrack n]}\in ext\left(
B_{\mathbb{R}^{n}}\right)  $, $i=1,\ldots,m$. Then
\[
g\cdot h = diag(x^{\mathbf{j}}y^{\mathbf{j}})_{\mathbf{j}\in[n]^{m}%
}=diag(z^{\mathbf{j}})_{\mathbf{j}\in\lbrack n]^{m}}\in G_{m}^{n}.
\]

(ii) Now let us show that $\phi$ is well defined, i.e., $\phi$ does not depend
on the representatives and $\phi\left(  G_{m}^{n},V_{m}^{n}\right)  $ is
contained in $V_{m}^{n}$.

Let us first show that $\phi$ does not depend on the representatives. Suppose
that $g\in G_{m}^{n}$ is represented by
\[
diag(a^{\mathbf{j}})_{\mathbf{j}\in\lbrack n]^{m}}=diag(b^{\mathbf{j}%
})_{\mathbf{j}\in\lbrack n]^{m}}%
\]
where $a_{1},\ldots,a_{m},b_{1},\ldots,b_{m}\in ext\left(  B_{\mathbb{R}^{n}%
}\right)  $ and $x_{1},\ldots,x_{m},y_{1},\ldots,y_{m}\in ext\left(
B_{\mathbb{R}^{n}}\right)  $ are such that
\[
\omega(x)=\omega(y).
\]
Then
\[
a^{\mathbf{j}} = b^{\mathbf{j}}\text{ and } x^{\mathbf{j}} = y^{\mathbf{j}}
\]
for all $\mathbf{j}\in[n]^{m}$. Thus,
\[
\omega(x)\cdot diag(a^{\mathbf{j}})_{\mathbf{j}\in[n]^{m}} = (x^{\mathbf{j}%
}a^{\mathbf{j}})_{\mathbf{j}\in[n]^{m}} = (y^{\mathbf{j}}b^{\mathbf{j}%
})_{\mathbf{j}\in[n]^{m}} = \omega(y)\cdot diag(b^{\mathbf{j}})_{\mathbf{j}%
\in[n]^{m}}.
\]
We conclude that $\phi$ does not depend on the representatives.

We will show that $\omega(x)\cdot g\in V_{m}^{n}$, where $x=(x_{1}%
,\ldots,x_{m})$ and $x_{1},\ldots,x_{m}\in ext(B_{\mathbb{R}^{n}})$. If
\[
g=diag(a^{\mathbf{j}})_{\mathbf{j}\in\lbrack n]^{m}}%
\]
with $a_{1},\ldots,a_{m}\in ext\left(  B_{\mathbb{R}^{n}}\right)  $, then
\[
\omega(x)\cdot g=(x^{\mathbf{j}}a^{\mathbf{j}})_{\mathbf{j}\in\lbrack n]^{m}%
}=\omega(y),
\]
with $y_{i}=(a_{i}^{(s)}x_{i}^{(s)})_{s\in\lbrack n]}\in ext(B_{\mathbb{R}%
^{n}})$, and thus $\omega(x)\cdot g\in V_{m}^{n}$.

(iii) Let us show that $\phi$ is a group action. Let $I$ be the identity of
$G_{m}^{n}$. Then
\[
\omega(x)\cdot I=\omega(x).
\]
Moreover, given $g,h\in G_{m}^{n}$, then
\begin{align*}
\phi(g\cdot h,\omega(x))  &  =\omega(x)\cdot(g\cdot h)\\
&  =\left(  \omega(x)\cdot g\right)  \cdot h\\
&  = \phi\left(  h,\omega(x)\cdot g \right) \\
&  =\phi\left(  h,\phi\left(  g,\omega(x)\right)  \right)  .
\end{align*}
Now let us show that $\phi$ is a free action. Given $x_{1},\ldots,x_{m}%
,y_{1},\ldots,y_{m}\in ext\left(  B_{\mathbb{R}^{n}}\right)  ,$ define
\[
g=diag(z^{\mathbf{j}})_{\mathbf{j}\in\lbrack n]^{m}},
\]
where $z_{i}=(x_{i}^{(s)}y_{i}^{(s)})_{s\in\lbrack n]}$. Thus
\[
\phi(g,\omega(x))=\omega(x)\cdot g=\omega(y).
\]

\end{proof}

\begin{corollary}
\label{novvv}Given $u\in V_{m}^{n},$ there is a basis of $\mathbb{R}^{n^{m}},$
$\beta\subseteq V_{m}^{n},$ such that $u\in\beta$. More precisely, if $\gamma$
is a basis of $\mathbb{R}^{n^{m}}$ contained in $V_{m}^{n},$ then
\[
\eta=\left\{  \phi\left(  g,v\right)  :v\in\gamma\right\}
\]
is a basis of $\mathbb{R}^{n^{m}}$ contained in $V_{m}^{n}$ for all $g\in$
$G_{m}^{n}.$
\end{corollary}

\begin{proof}
By Corollary \ref{coro1}, there is a basis $\varrho=\left\{  v_{1}%
,...,v_{n^{m}}\right\}  $ of $\mathbb{R}^{n^{m}}$ such that $\varrho\subseteq
V_{m}^{n}.$ Since $\phi:G_{m}^{n}\times V_{m}^{n}\rightarrow V_{m}^{n}$ is a
free (left) group action, there exists a $g_{1}\in G_{m}^{n}$ such that
\[
v_{1}\cdot g_{1}=\phi\left(  g_{1},v_{1}\right)  =u.
\]
Since
\[
v\cdot g=\phi\left(  g,v\right)  \in V_{m}^{n}%
\]
for all $g\in G_{m}^{n}$ and all $v\in V_{m}^{n}$ and since $g_{1}$ is
invertible,
\[
\beta:=\left\{  v\cdot g_{1}:v\in\varrho\right\}  =\left\{  \phi\left(
g_{1},v\right)  :v\in\varrho\right\}  \subseteq V_{m}^{n}%
\]
is a basis of $\mathbb{R}^{n^{m}}$ and obviously contains $u.$
\end{proof}

\section{The geometry of $\mathcal{L}(^{m}\mathbb{R}^{n})$}

\label{sct geo}

The main results of this section are Theorem \ref{arr} and Theorem \ref{arrimp}. They
provide an elementary constructive characterization of the extreme points of
the closed unit ball of $\mathcal{L}(^{m}\mathbb{R}^{n})$.

\subsection{The first main result}

Given a multilinear form $T\in\mathcal{L}(^{m}\mathbb{R}^{n})$, we can
represent it as%
\[
T(y)=\sum_{\mathbf{i}\in\lbrack n]^{m}}a_{\mathbf{i}}y^{\mathbf{i}},
\]
and thus%
\[
T(y)=\langle a^{T},\omega(y)\rangle,
\]
where
\[
a^{T}=(a_{\mathbf{i}})_{\mathbf{i}\in\lbrack n]^{m}}.
\]
For the sake of simplicity we shall sometimes denote $T$ just by $a^{T}$. The
following result is a straightforward consequence of the Krein-Milman Theorem:

\begin{proposition}
\label{8888}If $a^{\mathbf{T}}\in\mathcal{L}(^{m}\mathbb{R}^{n}),$ then
\[
\Vert a^{\mathbf{T}}\Vert=\max\left\{  |\langle a^{\mathbf{T}},\omega(x)
\rangle|:\ x =(x_{1},\ldots,x_{m})\text{ and } x_{1},\ldots,x_{m}\in
ext(B_{\mathbb{R}^{n}}) \right\}  .
\]

\end{proposition}

The following \textit{lemmata} can be easily verified and thus its proof omitted.

\begin{lemma}
\label{aux} Let $V$ be a vector space of dimension $m<\infty$. If
$\alpha=\{v_{1},\ldots,v_{k}\}$ is a linearly independent set of $V$ with
$k<m$ and $\beta=\{u_{1},\ldots,u_{m}\}$ is a basis of $V$, then there exists
a basis $\gamma$ of $V$ such that $\alpha\subseteq\gamma$ and $\gamma
\backslash\alpha\subseteq\beta$.
\end{lemma}

\begin{lemma}
\label{fim} Let $V$ be a vector space. If $\Omega=\left\{  v_{1}%
,...,v_{s}\right\}  $ is a set of non-null vectors in $V$, then there exists
$\alpha\subseteq\Omega$ such that $\alpha$ is a maximal linearly independent
set and
\[
\Omega\subseteq span\left(  \alpha\right)  .
\]

\end{lemma}

We will also use the following observation, which we announce as a lemma for
future reference:

\begin{lemma}
\label{lemapert} Let $v=(v_{1},\ldots,v_{n}),\ u=(u_{1},\ldots,u_{n})\text{
and }w=(w_{1},\ldots,w_{n})\in\mathbb{R}^{n}$. If
\[
w=\frac{1}{2}(u+v),
\]
then there is an $\alpha\in\mathbb{R}^{n}$ such that $u=w+\alpha$ and
$v=w-\alpha$.
\end{lemma}


Next is our first main result, which gives an instrumental characterization of
extreme points of the closed unit ball of the space of $m$-linear forms:

\begin{theorem}
\label{arr}Let $a^{T}\in B_{\mathcal{L}(^{m}\mathbb{R}^{n})}$. The following
assertions are equivalent:

\begin{itemize}
\item[(i)] $a^{T}\in ext\left(  B_{\mathcal{L}(^{m}\mathbb{R}^{n})}\right)  $

\item[(ii)] There exists $\beta\subseteq V_{m}^{n},$ basis of $\mathbb{R}%
^{n^{m}},$ such that $|\langle a^{\mathbf{T}},u\rangle|=1$ for all $u\in\beta$.
\end{itemize}
\end{theorem}

\begin{proof}
We start off by proving (ii) implies (i). Let $\beta\subseteq V_{m}^{n}$ be a
basis of $\mathbb{R}^{n^{m}},$ such that $|\langle a^{\mathbf{T}},u\rangle|=1$
for all $u\in\beta$. From Lemma \ref{lemapert} it suffices to prove that given
$b\in\mathbb{R}^{n^{m}}$ such that $a^{\mathbf{T}}+ b,a^{\mathbf{T}}- b\in
B_{\mathcal{L}(^{m}\mathbb{R}^{n})}$, we have $b=0$. If $a^{\mathbf{T}}+
b,a^{\mathbf{T}}- b\in B_{\mathcal{L}(^{m}\mathbb{R}^{n})}$, we have
\[
|\langle a^{\mathbf{T}}+ b,u\rangle|\leq1
\]
and
\[
|\langle a^{\mathbf{T}}- b,u\rangle|\leq1
\]
for all $u\in\beta$. Since $|\langle a^{\mathbf{T}},u\rangle|=1$ for all
$u\in\beta$, we have
\[
\langle b,u\rangle=0
\]
for all $u\in\beta$; therefore $b=0.$

Now let us prove that (i) implies (ii). Let us suppose, for the sake of
contradiction, that for all $\beta\subseteq V_{m}^{n},$ basis of
$\mathbb{R}^{n^{m}},$ there is $u_{\beta}\in\beta$ such that $|\langle
a^{\mathbf{T}},u_{\beta}\rangle|<1$. Note that
\[
\Omega:=\left\{  u\in V_{m}^{n}:|\langle a^{\mathbf{T}},u\rangle|=1\right\}
\]
does not contain any basis $\beta\subseteq V_{m}^{n}$ of $\mathbb{R}^{n^{m}}$
and $card(\Omega)<\infty$. Suppose that
\[
\Omega\neq\emptyset.
\]
Then, by Lemma \ref{fim} there is $\alpha=\{\eta_{1},\ldots,\eta
_{k}\}\subseteq\Omega$, a maximal linearly independent set, such that
\[
\Omega\subseteq span\left(  \alpha\right)  .
\]
By Lemma \ref{aux} and Corollary \ref{coro1}, there is a basis $\gamma
=\{\eta_{1},\ldots,\eta_{k},\xi_{1},\ldots,\xi_{n^{m}-k}\}$ of $\mathbb{R}%
^{n^{m}}$ contained in $V_{m}^{n},$ such that
\[
\xi_{j}\in V_{m}^{n}\backslash\Omega
\]
for all $j=1,...,n^{m}-k$. In fact, if there were $\xi_{j}\in\Omega$, then
$\xi_{j}\in span(\alpha)$ and $\gamma$ would be linearly dependent. Let
\[
V_{m}^{n}\backslash\Omega:=\{\zeta_{1},\ldots,\zeta_{s}\}.
\]
For all $i\in\lbrack s]$, there is a $r_{i}>0$ such that
\[
-1<\langle a^{\mathbf{T}},\zeta_{i}\rangle-r_{i}<\langle a^{\mathbf{T}}%
,\zeta_{i}\rangle<\langle a^{\mathbf{T}},\zeta_{i}\rangle+r_{i}<1.
\]
Defining $r=\min\{r_{i}:i\in\lbrack s]\}$, we have
\[
-1<\langle a^{\mathbf{T}},\zeta_{i}\rangle-r<\langle a^{\mathbf{T}},\zeta
_{i}\rangle<\langle a^{\mathbf{T}},\zeta_{i}\rangle+r<1.
\]
For all $i\in\lbrack s]$, there exist unique real scalars $p_{j}^{\zeta_{i}}$
and $l_{j}^{\zeta_{i}}$ such that
\[
\zeta_{i}=\sum_{j=1}^{n^{m}-k}p_{j}^{\zeta_{i}}\xi_{j}+\sum_{j=1}^{k}%
l_{j}^{\zeta_{i}}\eta_{j}.
\]
Define
\[
p:=\max\{|p_{1}^{\zeta_{i}}|:i=1,\ldots,s\}.
\]
Since $\xi_{1}\in V_{m}^{n}\backslash\Omega$, it follows that $p\geq1.$ Since
$\gamma$ is a basis, there is a $0\neq b\in\mathbb{R}^{n^{m}}$ such that
\[
\left(
\begin{array}
[c]{c}%
\xi_{1}\\
\xi_{2}\\
\vdots\\
\xi_{n^{m}-k}\\
\eta_{1}\\
\vdots\\
\eta_{k}%
\end{array}
\right)  b^{t}=\left(
\begin{array}
[c]{c}%
r/p\\
0\\
\vdots\\
0\\
0\\
\vdots\\
0
\end{array}
\right)  .
\]
Thus%
\begin{align*}
\langle\ a^{\mathbf{T}}\pm b,\zeta_{i}\rangle &  =\langle\ a^{\mathbf{T}%
},\zeta_{i}\rangle\pm\langle b,\zeta_{i}\rangle\\
&  =\langle a^{\mathbf{T}},\zeta_{i}\rangle\pm\langle b,\sum_{j=1}^{n^{m}%
-k}p_{j}^{\zeta_{i}}\xi_{j}+\sum_{j=1}^{k}l_{j}^{\zeta_{i}}\eta_{j}\rangle\\
&  =\langle a^{\mathbf{T}},\zeta_{i}\rangle\pm\langle b,p_{1}^{\zeta_{i}}%
\xi_{1}\rangle\\
&  =\langle a^{\mathbf{T}},\zeta_{i}\rangle\pm p_{1}^{\zeta_{i}}\langle
b,\xi_{1}\rangle\\
&  =\langle a^{\mathbf{T}},\zeta_{i}\rangle\pm\frac{p_{1}^{\zeta_{i}}r}{p}.
\end{align*}
Therefore
\[
\left\vert \langle a^{\mathbf{T}}\pm b,\zeta_{i}\rangle\right\vert =\left\vert
\langle a^{\mathbf{T}},\zeta_{i}\rangle\pm\frac{p_{1}^{\zeta_{i}}r}%
{p}\right\vert \leq\left\vert \langle a^{\mathbf{T}},\gamma_{i}\rangle\pm
r\right\vert <1.
\]
If $\upsilon\in\Omega$, then
\[
|\langle a^{\mathbf{T}}\pm b,\upsilon\rangle|=|\langle a^{T},\upsilon
\rangle|=1,
\]
because $\upsilon\in span(\alpha)$. We thus conclude that $a^{\mathbf{T}%
}+b,a^{\mathbf{T}}-b\in B_{\mathcal{L}(^{m}\mathbb{R}^{n})}.$ Since
\[
a^{\mathbf{T}}=\frac{1}{2}\left(  a^{\mathbf{T}}+b+a^{\mathbf{T}}-b\right)  ,
\]
it follows that $a^{\mathbf{T}}$ is not an extreme of $B_{\mathcal{L}%
(^{m}\mathbb{R}^{n})}$.

If $\Omega=\emptyset$, then $a^{\mathbf{T}}$ is an interior point of
$B_{\mathcal{L}(^{m}\mathbb{R}^{n})}$, and the proof is complete.
\end{proof}

\subsection{The second main result}

Let
\begin{equation}
\mathcal{B}=\{\beta_{1},\ldots,\beta_{s}\} \label{dezzz}%
\end{equation}
be the set of all basis of $\mathbb{R}^{n^{m}}$ such that $\beta_{j}\subseteq
V_{m}^{n}$ for all $j$ and $\omega(e,e,\ldots,e)\in\beta_{j}$ for all $j$,
where $e=e_{1}+\ldots+e_{n}$. By Corollary we have $\mathcal{B}\neq
\varnothing$. For all $i\in\lbrack s]$, define the matrix $H_{\beta_{i}}$
whose lines are the vectors of $\beta_{i}.$ For instance, if%
\[
\beta_{i}=\left\{  v_{i,1},...,v_{i,n^{m}}\right\}  ,
\]
then%
\begin{equation}
H_{\beta_{i}}=\left(
\begin{array}
[c]{c}%
v_{i,1}\\
\vdots\\
v_{i,n^{m}}%
\end{array}
\right)  \label{dezzz2}%
\end{equation}
is an $n^{m}\times n^{m}$ matrix. Consider, for all $i\in\lbrack s]$ and all
$f\in ext\left(  B_{\mathbb{R}^{n^{m}}}\right)  $, the sets%
\[
\mathcal{A}_{i,f}=\{a^{\mathbf{T}}:H_{\beta_{i}}(a^{\mathbf{T}})^{t}=f^{t}\}
\]
and%
\begin{equation}
\mathcal{A}=%
{\textstyle\bigcup\limits_{i\in\lbrack s],f\in ext\left(  B_{\mathbb{R}^{n}%
}\right)  }}
\mathcal{A}_{i,f}. \label{cal A}%
\end{equation}
Note that
\[
card(\mathcal{A})\leq card(\mathcal{B})\cdot2^{n^{m}}<\infty.
\]
Define, for all $g\in G_{m}^{n}$,
\[
\mathcal{C}_{g}=\left\{  a^{\mathbf{T}}\cdot g:a^{T}\in\mathcal{A},\ |\langle
a^{\mathbf{T}},v\rangle|\leq1\ \forall\ v\in V_{m}^{n}\text{ }\right\}
\label{conjut}%
\]
and%
\[
\mathcal{C}=%
{\textstyle\bigcup\limits_{g\in G_{m}^{n}}}
\mathcal{C}_{g}.
\]
Note also that
\[
card(\mathcal{C})\leq card(\mathcal{A})\cdot card(G_{m}^{n})<\infty.
\]

\begin{lemma}
\label{rel} Let $a^{\mathbf{T}}=(a_{\mathbf{i}})_{\mathbf{i}\in\lbrack n]^{m}%
}$, $g\in G_{m}^{n}$ and $\omega(x)$ with $x=(x_{1},\ldots,x_{m})$ and
$x_{1},\ldots,x_{m}\in ext(B_{\mathbb{R}^{n}})$. Then
\[
\langle a^{\mathbf{T}}\cdot g,\omega(x)\rangle=\langle a^{\mathbf{T}}%
,\phi(g,\omega(x))\rangle.
\]

\end{lemma}

\begin{proof}
Since $g\in G_{m}^{n}$, there are $b_{1},\ldots,b_{m}\in ext(B_{\mathbb{R}%
^{n}})$ such that
\[
g=diag(b^{\mathbf{j}})_{\mathbf{j}\in\lbrack n]^{m}}.
\]
Thus,
\[
a^{\mathbf{T}}\cdot g=(a_{\mathbf{i}})_{\mathbf{i}\in\lbrack n]^{m}}\cdot
diag(b^{\mathbf{j}})_{\mathbf{j}\in\lbrack n]^{m}}=(a_{\mathbf{j}%
}b^{\mathbf{j}})_{\mathbf{j}\in\lbrack n]^{m}}.
\]
Therefore
\begin{align*}
\langle a^{\mathbf{T}}\cdot g,\omega(x)\rangle &  =\sum_{\mathbf{j}\in\lbrack
n]^{m}}(a_{\mathbf{j}}b^{\mathbf{j}})x^{\mathbf{j}}\\
&  =\sum_{\mathbf{j}\in\lbrack n]^{m}}a_{\mathbf{j}}(x^{\mathbf{j}%
}b^{\mathbf{j}})\\
&  =\langle(a_{\mathbf{j}})_{\mathbf{j}\in\lbrack n]^{m}},(x^{\mathbf{j}%
}b^{\mathbf{j}})_{\mathbf{j}\in\lbrack n]^{m}}\rangle\\
&  =\langle a^{\mathbf{T}},\omega(x)\cdot g\rangle\\
&  =\langle a^{\mathbf{T}},\phi(g,\omega(x))\rangle.
\end{align*}

\end{proof}

Next theorem is our second main result of this section:

\begin{theorem}
\label{arrimp} $ext\left(  B_{\mathcal{L}(^{m}\mathbb{R}^{n})}\right)
=\mathcal{C}.$
\end{theorem}

\begin{proof}
Let us first show that $ext\left(  B_{\mathcal{L}(^{m}\mathbb{R}^{n})}\right)
\subseteq\mathcal{C}$. If $a^{\mathbf{T}}\in ext\left(  B_{\mathcal{L}%
(^{m}\mathbb{R}^{n})}\right)  $, then by Theorem \ref{arr} there exists
$\beta=\{v_{1},\ldots,v_{n^{m}}\}\subseteq V_{m}^{n},$ basis of $\mathbb{R}%
^{n^{m}},$ such that
\[
|\langle a^{\mathbf{T}},v\rangle|=1\ \forall\ v\in\beta.
\]
Let $H$ be the matrix whose lines are the vectors of $\beta$ and let%
\[
f=(\langle a^{\mathbf{T}},v_{1}\rangle,\ldots,\langle a^{\mathbf{T}},v_{n^{m}%
}\rangle)\in ext\left(  B_{\mathbb{R}^{n^{m}}}\right)  .
\]
Then
\[
H\cdot(a^{\mathbf{T}})^{t}=f^{t}.
\]
Since $\phi:G_{m}^{n}\times V_{m}^{n}\rightarrow V_{m}^{n}$ is a free action,
there is a $g\in G_{m}^{n}$ such that $\phi\left(  g,v_{1}\right)
=\omega(e,e,\ldots,e)$. Then, still using the notation introduced in
(\ref{dezzz}), by Corollary \ref{novvv} we have
\[
\left\{  \phi\left(  g,v\right)  :v\in\beta\right\}  =\beta_{j}%
\]
for a certain $j.$ Therefore
\[
H_{\beta_{j}}=\left(
\begin{array}
[c]{c}%
\phi\left(  g,v_{1}\right) \\
\vdots\\
\phi\left(  g,v_{n^{m}}\right)
\end{array}
\right)  =H\cdot g.
\]
Let $x$ be solution of%
\[
H_{\beta_{j}}\cdot x^{t}=f^{t}.
\]
Hence
\begin{align*}
x^{t}  &  =H_{\beta_{j}}^{-1}\cdot f^{t}\\
&  =g^{-1}\cdot\left(  H^{-1}\cdot f^{t}\right) \\
&  =g\cdot(a^{\mathbf{T}})^{t}.
\end{align*}
Therefore,
\begin{equation}
a^{\mathbf{T}}\cdot g=x\in\mathcal{A}_{j,f}\subseteq\mathcal{A}. \label{ummm}%
\end{equation}
Given $y=(y_{1},\ldots,y_{m})$ with $y_{1},\ldots,y_{m}\in ext(B_{\mathbb{R}%
^{n}})$, by Lemma \ref{rel}, we have
\begin{equation}
|\langle x,\omega(y)\rangle|=|\langle a^{\mathbf{T}}\cdot g,\omega
(y)\rangle|=|\langle a^{\mathbf{T}},\phi(g,\omega(y))\rangle|\leq1.
\label{doisss}%
\end{equation}
By Proposition \ref{8888} we have $x\in B_{\mathcal{L}(^{m}\mathbb{R}^{n})}$
and finally, by (\ref{ummm}) and (\ref{doisss}) we get
\[
a^{\mathbf{T}}=x\cdot g\in\mathcal{C}_{g}\subseteq\mathcal{C}\text{,}%
\]
i.e.,%
\[
ext\left(  B_{\mathcal{L}(^{m}\mathbb{R}^{n})}\right)  \subseteq\mathcal{C}.
\]
Now, let us show that $\mathcal{C}\subseteq ext\left(  B_{\mathcal{L}%
(^{m}\mathbb{R}^{n})}\right)  $. If $x\in\mathcal{C}$, then $x=a^{\mathbf{T}%
}\cdot g$ with $a^{\mathbf{T}}\in\mathcal{A}$ and $g\in G_{m}^{n}$, where
$|\langle a^{\mathbf{T}},v\rangle|\leq1$ for all $v\in V_{m}^{n}$. There are
$j$ and $f$ such that $a^{\mathbf{T}}\in\mathcal{A}_{j,f}$ and there is
$\beta_{j}\in\mathcal{B}$ such that $|\langle a^{\mathbf{T}},u\rangle|=1$ for
all $u\in\beta_{j}$. Since $g$ is invertible then
\[
\beta:=\{\phi(g,u):u\in\beta_{j}\}
\]
is a basis of $\mathbb{R}^{n^{m}}$. Then, for all $u\in\beta_{j}$, by Lemma
\ref{rel}, we have
\[
|\langle x,\phi(g,u)\rangle|=|\langle(a^{\mathbf{T}}\cdot g),\phi
(g,u)\rangle|=|\langle a^{\mathbf{T}},\phi(g\cdot g,u)\rangle|=|\langle
a^{\mathbf{T}},u\rangle|=1.
\]
Since $|\langle a^{\mathbf{T}},v\rangle|\leq1$ for all $v\in V_{m}^{n}$, using
the same argument, by Lemma \ref{rel}, we have
\[
|\langle x,\phi(g,v)\rangle|=|\langle(a^{\mathbf{T}}\cdot g),\phi
(g,v)\rangle|=|\langle a^{\mathbf{T}},\phi(g\cdot g,v)\rangle|=|\langle
a^{\mathbf{T}},v\rangle|\leq1
\]
for all $v\in V_{m}^{n}$ and hence $x\in B_{\mathcal{L}(^{m}\mathbb{R}^{n})}$.
By Theorem \ref{arr} we conclude that $x\in ext\left(  B_{\mathcal{L}%
(^{m}\mathbb{R}^{n})}\right)  $.
\end{proof}

\begin{corollary}
For all positive integers $m,n$, the coefficients of the extreme points $T\in
B_{\mathcal{L}(^{m}\mathbb{R}^{n})}$ are rational numbers.
\end{corollary}

\begin{proof}
Note that we start off with an $n^{m}\times n^{m}$ matrix whose entries are
$1$ or $-1.$ We solve a linear system whose independent terms are $1$ or $-1.$
The extreme points are found among these solutions, and obviously all of its
coordinates are rational numbers.
\end{proof}

\section{Constructive process}

\label{CPPP}

An easily implemented algorithm can be extracted from the proofs delivered in
the previous two sections. Below we summarize how to find all extreme points
of the closed unit ball of $B_{\mathcal{L}(^{m}\mathbb{R}^{n})}$:

\begin{itemize}
\item[Step 1:] Determinate all $n^{m}\times n^{m}$ invertible matrices whose
lines belong to $V_{m}^{n},$ that contain $\omega(e,\ldots,e)$. Note that
using the notations from (\ref{dezzz}) and (\ref{dezzz2}) the set of such
matrices is%
\[
\mathcal{M}=\left\{  H_{\beta_{i}}:\beta_{i}\in\mathcal{D}\subseteq
\mathcal{B}\right\}
\]
for a certain $\mathcal{D}$.

\item[Step 2:] For all choices of $f\in ext({B_{\mathbb{R}^{n^{m}}}})$ and
each matrix $H_{\beta_{i}}$ collected in Step 1, solve the linear system%
\[
H_{\beta_{i}}\cdot\left(  a^{\mathbf{T}}\right)  ^{t}=f^{t}.
\]

\item[Step 3:] Among all solutions given by the second step, verify which
solutions also satisfy
\[
|\langle a^{\mathbf{T}},v\rangle|\leq1
\]
for all $v\in V_{m}^{n}$.

\item[Step 4:] Among all solutions given by the third step, calculate
\[
a^{\mathbf{T}}\cdot g
\]
for all $g\in G_{m}^{n}$. The set of all such $a^{\mathbf{T}}\cdot g$ is
precisely the set of all extreme points of $B_{\mathcal{L}(^{m}\mathbb{R}%
^{n})}.$
\end{itemize}

\subsection{Examples}

As mentioned earlier, previous knowledge on extreme points of the unit ball in
the space of multilinear forms were limited to low dimensions and/or low
degrees. The simplest case, $n=m=2$, appears in the work of S.G. Kim,
\cite{kim}, and accordingly can be obtained by our method.

\begin{example}
\label{jaa}All extreme points of $B_{\mathcal{L}(^{2}\mathbb{R}^{2})}$ are:
\[%
\begin{array}
[c]{cccc}%
\pm(0,0,0,1), & \pm\frac{1}{2}(1,1,1,-1), & \pm\frac{1}{2}(1,1,-1,1), &
\pm\frac{1}{2}(1,-1,1,1),\\
\pm\frac{1}{2}(-1,1,1,1), & \pm(0,0,1,0), & \pm(0,1,0,0), & \pm(1,0,0,0).
\end{array}
\]

\end{example}

For 3-forms and 4-forms, though, very little, if anything, were previously
known, even restricted to the plane. Here are some illustrative examples:

\begin{example}
The following vectors are extreme points of $B_{\mathcal{L}(^{3}\mathbb{R}%
^{2})}$:
\[%
\begin{array}
[c]{ccc}%
\pm(1,0,0,0,0,0,0,0), & \pm\frac{1}{4}(1,-1,-1,1,-1,1,1,3), & \pm\frac{1}%
{2}(0,0,0,0,-1,1,1,1).
\end{array}
\]
All extreme points of $B_{\mathcal{L}(^{3}\ell_{\infty}^{2})}$ can be found
through the algorithm above described.
\end{example}

\begin{example}
Here are some extreme points of $B_{\mathcal{L}(^{4}\mathbb{R}^{2})}$:
\begin{align*}
&  \pm(1,0,0,0,0,0,0,0,0,0,0,0,0,0,0,0),\\
&  \pm\frac{1}{8}(-1,1,1,-1,1,-1,-1,1,1,-1,-1,1,-1,1,1,7),\\
&  \pm\frac{1}{4}(0,1,0,1,0,-1,0,1,0,-1,0,1,0,1,0,3),\\
&  \pm\frac{1}{8}(1,1,1,-3,-1,-1,-1,3,-1,-1,-1,3,1,1,1,5,)\\
&  \pm\frac{1}{2}(0,0,0,-1,0,0,0,1,0,0,0,1,0,0,0,1),\\
&  \pm\frac{1}{4}(0,0,-1,-1,1,-1,0,2,0,0,1,1,-1,1,0,2).
\end{align*}
Again, the complete list of extreme points of $B_{\mathcal{L}(^{3}%
\mathbb{R}^{2})}$ can be found through the algorithm above described.
\end{example}

\subsection{The planar case}

In the special case, $n=2$, we have
\[
\Vert a^{T}\Vert=\max\left\{  |\langle a^{T},\omega(x_{1},\ldots,x_{m}%
)\rangle|:x_{1},\ldots,x_{m}\in\{(1,1),(-1,1)\}\right\} ,
\]
for any arbitrary integer $m$. Let $x=(x_{i})_{i=1}^{m}$ and $y=(y_{i}%
)_{i=1}^{m}$ be such that
\[
x_{i},y_{i}\in\{(1,1),(-1,1)\}
\]
for all $i\in\lbrack m]$. Since $x\neq y$, there exists a $j_{0}\in\lbrack m]$
such that $x_{j_{0}}\neq y_{j_{0}}$. Thus, by Lemma 1, we have
\[
\langle\omega(x),\omega(y)\rangle=\Pi_{i=1}^{m}\langle x_{i},y_{i}\rangle=0.
\]
So, we can determinate the extreme points of $B_{\mathcal{L}(^{m}%
\mathbb{R}^{2})}$ as follows:

Step 1: Build the matrix $H$ such that the lines are the values of
$\Lambda_{2}(\beta)$, where $\beta=\{(1,1),(-1,1)\}$.

Step 2: For each matrix $f\in ext(B_{\mathbb{R}^{2^{m}}})$, solve the linear
system
\[
Hx^{t}=f^{t}.
\]

From the above routine we have the following result:

\begin{proposition}
For all positive integer $m$ we have%
\[
card(ext(B_{\mathcal{L}(^{m}\mathbb{R}^{2})}))=2^{(2^{m})}.
\]

\end{proposition}

\section{Applications: optimization problems in classical
inequalities\label{apli}}

{In this section we briefly discuss the fit of our main characterization theorems within investigations pertaining to 
classical inequalities. Of particular interest, we
formally solve the open problem of determining all optimal constants
of the $m$--linear Bohnenblust--Hille inequalities for real scalars.}

We start off with two observations, which we state as propositions for future references. The former is a straightforward consequence of the Krein-Milman Theorem (Lemma
\ref{km}) and Theorem \ref{arrimp}:

\begin{proposition}
\label{arrmax}Let $f \colon B_{\mathcal{L}\left(  ^{m}\mathbb{R}^{n}\right)
}\rightarrow\mathbb{R}$ be a convex and continuous function. Then%
\[
\max_{T\in B_{\mathcal{L}\left(  ^{m}\mathbb{R}^{n}\right)  }}f(T)=\max
\left\{  f(T):T\in\mathcal{C}\right\}  .
\]
\end{proposition}

The next result is also useful for computational purposes:

\begin{proposition}
Let $1\leq\lambda<\infty$. If $f_{\lambda} \colon B_{\mathcal{L}(^{m}\mathbb{R}^{n}%
)}\rightarrow\mathbb{R}$ is
\[
f_{\lambda}(T)=\left(  \sum_{\mathbf{i}\in\lbrack n]^{m}}\left\vert
T(e_{i_{1}},\ldots,e_{i_{m}})\right\vert ^{\lambda}\right)  ^{\frac{1}%
{\lambda}},
\]
then
\[
\max_{T\in B_{\mathcal{L}(^{m}\mathbb{R}^{n})}}f_{\lambda}(T)=\max\left\{
f_{\lambda}(T):T\in\mathcal{A}\cap B_{\mathcal{L}(^{m}\mathbb{R}^{n}%
)}\right\}  ,
\]
where $\mathcal{A}$ is given by \eqref{cal A}.
\end{proposition}

\begin{proof}
By Proposition \ref{arrmax} we know that%
\[
\max_{T\in B_{\mathcal{L}\left(  ^{m}\mathbb{R}^{n}\right)  }}f_{\lambda
}(T)=\max\left\{  f_{\lambda}(T):T\in\mathcal{C}\right\}  ,
\]
i.e., there is a $T_{0}\in\mathcal{C}$ such that%
\[
\max_{T\in B_{\mathcal{L}\left(  ^{m}\mathbb{R}^{n}\right)  }}f_{\lambda
}(T)=f_{\lambda}(T_{0}).
\]
Since $T_{0}=a^{S}\cdot g$ for some $a^{S}\in\mathcal{A}$ and $g\in G_{m}^{n}$
and
\[
|\langle a^{S},v\rangle|\leq1
\]
for all $v\in V_{m}^{n}$. We have%
\begin{align*}
\max\left\{  f_{\lambda}(T):T\in\mathcal{A}\cap B_{\mathcal{L}(^{m}%
\mathbb{R}^{n})}\right\}   &  \leq\max\left\{  f_{\lambda}(T):T\in
\mathcal{C}\right\} \\
&  =f_{\lambda}(T_{0})\\
&  =f_{\lambda}(a^{S}\cdot g)\\
&  =f_{\lambda}(a^{S})\\
&  \leq\max\left\{  f_{\lambda}(T):T\in\mathcal{A}\cap B_{\mathcal{L}%
(^{m}\mathbb{R}^{n})}\right\}  .
\end{align*}

\end{proof}

\subsection{Classical multilinear inequalities: sharp values}

Let $\mathbb{K}=\mathbb{R}$ or $\mathbb{C}.$ The (classical)
Bohnenblust--Hille inequality, \cite{bh}, asserts that for all $m$--linear
forms $T\colon\mathbb{K}^{n}\times\cdots\times\mathbb{K}^{n}\rightarrow
\mathbb{K}$ and all positive integers $n$,
\begin{equation}
\textstyle\left(  \sum\limits_{j_{1},...,j_{m}=1}^{n}\left\vert T(e_{j_{1}%
},...,e_{j_{m}})\right\vert ^{\frac{2m}{m+1}}\right)  ^{\frac{m+1}{2m}}\leq
B_{m}^{\mathbb{K}}(n)\left\Vert T\right\Vert ,\label{tttt}%
\end{equation}
for an optimal constant $B_{m}^{\mathbb{K}}(n)\geq1,$ and%
\begin{equation}
B_{m}^{\mathbb{K}}(\infty):=\sup_{n}B_{m}^{\mathbb{K}}(n)<\infty.\label{7777}%
\end{equation}

From Proposition \ref{arrmax} we have the following formula for the optimal
constants $B_{m}^{\mathbb{R}}(n):$%

\begin{equation}
B_{m}^{\mathbb{R}}(n)=\max\left\{  T\in\mathcal{C}_{m,n}\right\}  ,
\label{8mz}%
\end{equation}
where $\mathcal{C}_{m,n}$ is the (finite) set created by the elementary
constructive process of Section \ref{CPPP}.

When $m=2$, inequality (\ref{tttt}) recovers the famous Littlewood's $4/3$
inequality, and it is well known that $B_{2}^{\mathbb{R}}(\infty)=\sqrt{2}$.
For $m\geq3$, the precise values of sharp constants $B_{m}^{\mathbb{K}}%
(\infty)$ remain unknown, despite of their intrinsic applications in the case
of real scalars, see \cite{montanaro}.

It follows from (\ref{8mz}), however, that given two positive integers $m,n$
the precise value of $B_{m}^{\mathbb{R}}(n)$ can be fully determined and
formally computed by the constructive method earlier described after a finite
number of elementary steps. The same can be done for any similar inequalities,
like the mixed Littlewood-type inequalities.

\subsection{Classical multilinear inequalities: algebraic properties}

It is appealing to observe that, since the coordinates of extreme points of
$B_{\mathcal{L}\left(  ^{m}\mathbb{R}^{n}\right)  }$ are rational numbers, we
can easily conclude that:

\begin{proposition}
For all positive integers $m,n$, the optimal constants $B_{m}^{\mathbb{R}}(n)$
are algebraic numbers.
\end{proposition}

The above result cannot be straightforwardly extended to the case $n=\infty,$
i.e., we cannot conclude $B_{m}^{\mathbb{R}}(\infty)$ are algebraic numbers.
In what follows though, we will show that when considering $m$-linear forms
defined over $\mathbb{R}^{n}\times\mathbb{R}^{n}\times\mathbb{R}^{2}%
\times\cdots\times\mathbb{R}^{2}$ with $n\geq2^{m-1}$ the sharp constants are
indeed algebraic, and equal $2^{1-\frac{1}{m}}.$ This result provides a
partial solution to the question of whether the constants $B_{m}^{\mathbb{R}%
}(\infty)$ are algebraic or not. Our proof is based on a different set of
tools, which includes the Mixed Littlewood inequality and the Khinchin
inequality; we recall them here for the sake of the readers:

\noindent\textbf{Mixed Littlewood inequality.} For all continuous
$(m+1)$--linear forms $T\colon\mathbb{R}^{n}\times\cdots\times\mathbb{R}%
^{n}\rightarrow\mathbb{R}$ and for all positive integers $n,$ we have%
\begin{equation}
\left(  \sum\limits_{j_{1},...,j_{m}=1}^{n}\left(  \sum\limits_{j_{m+1}=1}%
^{n}\left\vert T(e_{j_{1}},...,e_{j_{m}})\right\vert ^{2}\right)  ^{\frac
{1}{2}\frac{2m}{m+1}}\right)  ^{\frac{m+1}{2m}}\leq M_{m+1}(n)\left\Vert
T\right\Vert \label{hhb}%
\end{equation}
and%
\[
M_{m+1}:=\sup_{n}M_{m+1}(n)<\infty.
\]

\noindent\textbf{Khinchin inequality.}(see \cite{diestel}). For any
$0<q<\infty$, there are positive constants $A_{q}$, $B_{q}$ such that
\[
A_{q}\left(  \sum\limits_{j=1}^{n}|a_{j}|^{2}\right)  ^{\frac{1}{2}}%
\leq\left(  \int_{0}^{1}\left\vert \sum\limits_{j=1}^{n}a_{j}r_{j}%
(t)\right\vert ^{q}dt\right)  ^{\frac{1}{q}}\leq B_{q}\left(  \sum
\limits_{j=1}^{n}|a_{j}|^{2}\right)  ^{\frac{1}{2}},
\]
for any positive integer $n$ and sequence of scalars $(a_{j})_{j=1}^{n}$. Here
$r_{j}$ denote the Rademacher functions. The best constants $A_{q}$ are (see
\cite{diestel}):%
\begin{align*}
A_{q} &  =\sqrt{2}\left(  \frac{\Gamma\left(  \frac{1+q}{2}\right)  }%
{\sqrt{\pi}}\right)  ^{\frac{1}{q}}\text{ if }2>q\geq q_{0}\cong1.8474;\\
A_{q} &  =2^{\frac{1}{2}-\frac{1}{q}}\text{ if }q<q_{0}.
\end{align*}
The number $q_{0}$ above is the unique real scalar satisfying $\Gamma\left(
\frac{q_{0}+1}{2}\right)  =\frac{\sqrt{\pi}}{2}$.\medskip

\begin{lemma}
\label{t11}Let $m\geq1$ and $n\geq2$ be positive integers. For all continuous
$(m+1)$--linear forms $T:\mathbb{R}^{n}\times\cdots\times\mathbb{R}^{n}%
\times\mathbb{R}^{2}\rightarrow\mathbb{R}$ we have%
\begin{equation}
\left(  \sum\limits_{j_{1},...,j_{m}=1}^{n}\left(  \sum\limits_{j_{m+1}=1}%
^{2}\left\vert T(e_{j_{1}},...,e_{j_{m}})\right\vert ^{2}\right)  ^{\frac
{1}{2}\frac{2m}{m+1}}\right)  ^{\frac{m+1}{2m}}\leq2^{\frac{1}{2m}}%
B_{m}(n)\left\Vert T\right\Vert \label{u43}%
\end{equation}
and the constant $2^{\frac{1}{2m}}B_{m}(n)$ is sharp.
\end{lemma}

\begin{proof}
The inequality%
\begin{equation}
\left(  \sum\limits_{j_{1},...,j_{m}=1}^{n}\left(  \sum\limits_{j_{m+1}=1}%
^{2}\left\vert T(e_{j_{1}},...,e_{j_{m}})\right\vert ^{2}\right)  ^{\frac
{1}{2}\frac{2m}{m+1}}\right)  ^{\frac{m+1}{2m}}\leq A_{\frac{2m}{m+1}}%
^{-1}B_{m}(n)\left\Vert T\right\Vert \label{est}%
\end{equation}
is a straightforward consequence of the Khinchin inequality; here
$A_{\frac{2m}{m+1}}$ are the associated constants of the Khinchin inequality.
Since for any $1\leq p\leq2$ the maximum of%
\[
f(a,b)=\frac{\left(  a^{2}+b^{2}\right)  ^{1/2}}{\left(  \frac{1}{2}\left\vert
a+b\right\vert ^{p}+\frac{1}{2}\left\vert a-b\right\vert ^{p}\right)  ^{1/p}}%
\]
is $2^{\frac{1}{p}-\frac{1}{2}},$ in our case the constants of the Khinchin
inequality can be taken as $2^{\frac{m+1}{2m}-\frac{1}{2}},$ i.e.,
$2^{\frac{1}{2m}}$ (recall that we are dealing with continuous $m+1$--linear
forms $T:\mathbb{R}^{n}\times\cdots\times\mathbb{R}^{n}\times\mathbb{R}%
^{2}\rightarrow\mathbb{R}$). Thus
\[
\left(  \sum\limits_{j_{1},...,j_{m}=1}^{n}\left(  \sum\limits_{j_{m+1}=1}%
^{2}\left\vert T(e_{j_{1}},...,e_{j_{m}})\right\vert ^{2}\right)  ^{\frac
{1}{2}\frac{2m}{m+1}}\right)  ^{\frac{m+1}{2m}}\leq2^{\frac{1}{2m}}%
B_{m}(n)\left\Vert T\right\Vert .
\]
We just need to prove that the constant $2^{\frac{1}{2m}}B_{m}(n)$ is sharp.

From now on, for any continuous $m$-linear form $T_{m}:\mathbb{R}^{n}%
\times\cdots\times\mathbb{R}^{n}\rightarrow\mathbb{R}$ we define%
\begin{align*}
\widetilde{T_{m}}(x^{(1)},...,x^{(m)})  &  =T_{m}(z^{(1)},...,z^{(m)}),\\
\widetilde{\widetilde{T_{m}}}(x^{(1)},...,x^{(m)})  &  =T_{m}(w^{(1)}%
,...,w^{(m)}),
\end{align*}
where, for all $k=1,....,m,$ we consider%
\begin{align*}
z^{(k)}  &  =(x_{1}^{(k)},x_{3}^{(k)},x_{5}^{(k)},...),\\
w^{(k)}  &  =(x_{2}^{(k)},x_{4}^{(k)},x_{6}^{(k)},...).
\end{align*}
Note that%
\[
\left\Vert \widetilde{\widetilde{T_{m}}}\right\Vert =\left\Vert \widetilde
{T_{m}}\right\Vert =\left\Vert T_{m}\right\Vert .
\]

Let $\varepsilon>0$ and $T_{m}:\mathbb{R}^{n}\times\cdots\times\mathbb{R}%
^{n}\rightarrow\mathbb{R}$ be such that%
\begin{equation}
\left(  \sum\limits_{j_{1},...,j_{m}=1}^{n}\left\vert T_{m}(e_{j_{1}%
},...,e_{j_{m}})\right\vert ^{\frac{2m}{m+1}}\right)  ^{\frac{m+1}{2m}%
}>\left(  B_{m}^{\mathbb{R}}(n)-\varepsilon\right)  \left\Vert T_{m}%
\right\Vert . \label{1122}%
\end{equation}
Define the $m+1$-linear operator $R_{m+1}:\mathbb{R}^{n}\times\cdots
\times\mathbb{R}^{n}\times\mathbb{R}^{2}\rightarrow\mathbb{R}$ by%
\begin{align*}
R_{m+1}(x^{(1)},...,x^{(m+1)})  &  =\left(  x_{2}^{(m+1)}-x_{1}^{(m+1)}%
\right)  \widetilde{T_{m}}(x^{(1)},...,x^{(m)})\\
&  +\left(  x_{2}^{(m+1)}+x_{1}^{(m+1)}\right)  \widetilde{\widetilde{T_{m}}%
}(x^{(1)},...,x^{(m)}).
\end{align*}
By the definition of $R_{m+1}$, we have%
\[
\left\Vert R_{m+1}\right\Vert =\left\Vert 2T_{m}\right\Vert
\]
and we can also note that for all $e_{j_{1},...,}e_{j_{m}},$ we have%
\[
\left\vert R_{m+1}(e_{j_{1}},...,e_{j_{m}},e_{1})\right\vert =\left\vert
R_{m+1}(e_{j_{1}},...,e_{j_{m}},e_{2})\right\vert .
\]
Since we are using just two coordinates of the last variable and since, for
any $1\leq p\leq2$, the maximum of%
\[
f(a,b)=\frac{\left(  a^{2}+b^{2}\right)  ^{1/2}}{\left(  \frac{1}{2}\left\vert
a+b\right\vert ^{p}+\frac{1}{2}\left\vert a-b\right\vert ^{p}\right)  ^{1/p}}%
\]
is $2^{\frac{1}{p}-\frac{1}{2}}$ (it is attained when $\left\vert a\right\vert
=\left\vert b\right\vert >0$), we have%
\begin{align*}
&  \left(  \sum\limits_{j_{1},...,j_{m}=1}^{n}\left(  \sum\limits_{j_{m+1}%
=1}^{2}\left\vert R_{m+1}(e_{j_{1}},...,e_{j_{m+1}})\right\vert ^{2}\right)
^{\frac{1}{2}\frac{2m}{m+1}}\right)  ^{\frac{m+1}{2m}}\\
&  =2^{\frac{m+1}{2m}-\frac{1}{2}}\left(  \sum\limits_{j_{1},...,j_{m}=1}%
^{n}\sum\limits_{\left(  \varepsilon_{1},\varepsilon_{2}\right)
\in\{(1,-1),(1,1)\}}\frac{1}{2}\left\vert R_{m+1}(e_{j_{1}},...,e_{j_{m}%
},\varepsilon_{1}e_{1}+\varepsilon_{2}e_{2})\right\vert ^{\frac{2m}{m+1}%
}\right)  ^{\frac{m+1}{2m}}\\
&  =2^{\frac{1}{2m}}\left(  \frac{1}{2}\sum\limits_{j_{1},...,j_{m}=1}%
^{n}\left\vert 2\widetilde{\widetilde{T_{m}}}(e_{j_{1}},...,e_{j_{m}%
})\right\vert ^{\frac{2m}{m+1}}+\frac{1}{2}\sum\limits_{j_{1},...,j_{m}=1}%
^{n}\left\vert 2\widetilde{T_{m}}(e_{j_{1}},...,e_{j_{m}})\right\vert
^{\frac{2m}{m+1}}\right)  ^{\frac{m+1}{2m}}.
\end{align*}
It is obvious that both $2\widetilde{\widetilde{T_{m}}}$ and $2\widetilde
{T_{m}}$ also satisfy (\ref{1122}). Thus%
\begin{align*}
&  \left(  \sum\limits_{j_{1},...,j_{m}=1}^{n}\left(  \sum\limits_{j_{m+1}%
=1}^{2}\left\vert R_{m+1}(e_{j_{1}},...,e_{j_{m+1}})\right\vert ^{2}\right)
^{\frac{1}{2}\frac{2m}{m+1}}\right)  ^{\frac{m+1}{2m}}\\
&  >2^{\frac{1}{2m}}\left(  \frac{1}{2}\left(  \left(  B_{m}^{\mathbb{R}%
}(n)-\varepsilon\right)  \left\Vert 2T_{m}\right\Vert \right)  ^{\frac
{2m}{m+1}}+\frac{1}{2}\left(  \left(  B_{m}^{\mathbb{R}}(n)-\varepsilon
\right)  \left\Vert 2T_{m}\right\Vert \right)  ^{\frac{2m}{m+1}}\right)
^{\frac{m+1}{2m}}\\
&  =2^{\frac{1}{2m}}\left(  \frac{1}{2}\left(  \left(  B_{m}^{\mathbb{R}%
}(n)-\varepsilon\right)  \left\Vert R_{m+1}\right\Vert \right)  ^{\frac
{2m}{m+1}}+\frac{1}{2}\left(  \left(  B_{m}^{\mathbb{R}}(n)-\varepsilon
\right)  \left\Vert R_{m+1}\right\Vert \right)  ^{\frac{2m}{m+1}}\right)
^{\frac{m+1}{2m}}\\
&  =2^{\frac{1}{2m}}\left(  B_{m}^{\mathbb{R}}(n)-\varepsilon\right)
\left\Vert R_{m+1}\right\Vert .
\end{align*}
Letting $\varepsilon\rightarrow0$ we thus conclude that $2^{\frac{1}{2m}}%
B_{m}^{\mathbb{R}}(n)$ is sharp.
\end{proof}

Suppose that now we have $m$-linear forms defined in $\mathbb{R}^{n}%
\times\mathbb{R}^{n}\times\mathbb{R}^{2}\times\cdots\times\mathbb{R}^{2}$ with
$n\geq2^{m-1}.$ The proof that the sharp constants are $2^{1-\frac{1}{m}}$ is
now a straightforward consequence of the H\"{o}lder inequality for mixed sums
combined with (\ref{u43}) and the following simple inequality:%
\begin{equation}
\left(  \sum\limits_{j_{1},...,j_{m-1}}\left(  \sum\limits_{j_{m}=1}%
^{2}\left\vert T(e_{j_{1}},...,e_{j_{m}})\right\vert ^{1}\right)  ^{\frac
{1}{1}2}\right)  ^{\frac{1}{2}}\leq\sqrt{2}\left\Vert T\right\Vert \label{i8}%
\end{equation}
for all $T:\mathbb{R}^{n_{1}}\times\cdots\times\mathbb{R}^{n_{m-1}}%
\times\mathbb{R}^{2}\rightarrow\mathbb{R}$. Considering the strongly
non-symmetric $m$-linear forms used in the proof of \cite[Theorem 4.1]{pt} we
easily prove that the estimates are sharp.

\subsection{The case of complex scalars\label{fini}}

The case of the optimal Bohnenblust--Hille constants for complex scalars is
obviously not encompassed by the previous techniques. The main point is that
the geometry of the closed unit ball $B_{\mathcal{L}\left(  ^{m}\mathbb{C}%
^{n}\right)  }$ is rather different and essentially unknown. In this
subsection, however, we tackle R. Blei's problem concerning sharp estimates
for complex inequalities; more precisely, Orlicz's, Littlewood's $\left(
\ell_{1},\ell_{2}\right)  $, and Littlewood's $4/3$ inequalities:

For each positive integer $n$, following the notation used by \cite{blei}, let
$\kappa_{O}^{\mathbb{K}}(n)$, $\kappa_{L}^{\mathbb{K}}(n)$, $\kappa
_{4/3}^{\mathbb{K}}(n)$ be extrema constants for the following inequalities:%
\[
\left(  \sum\limits_{i=1}^{n}\left(  \sum\limits_{j=1}^{n}\left\vert
T(e_{i},e_{j})\right\vert \right)  ^{2}\right)  ^{1/2}\leq\kappa
_{O}^{\mathbb{K}}(n)\left\Vert T\right\Vert ,
\]%
\[
\sum\limits_{i=1}^{n}\left(  \sum\limits_{j=1}^{n}\left\vert T(e_{i}%
,e_{j})\right\vert ^{2}\right)  ^{1/2}\leq\kappa_{L}^{\mathbb{K}}(n)\left\Vert
T\right\Vert ,
\]
and
\[
\left(  \sum\limits_{i,j=1}^{n}\left\vert T(e_{i},e_{j})\right\vert
^{4/3}\right)  ^{3/4}\leq\kappa_{4/3}^{\mathbb{K}}(n)\left\Vert T\right\Vert
\]
for all bilinear forms $T\colon\mathbb{C}^{n}\times\mathbb{C}^{n}%
\rightarrow\mathbb{C}$. Classical inequalities, see \cite{blei, Lit, or}, due
to Orlicz and Littlewood assert that
\begin{align*}
\kappa_{O}^{\mathbb{C}}(\infty)  &  :=\lim_{n\rightarrow\infty}\kappa
_{O}^{\mathbb{C}}(n)<\infty,\\
\kappa_{L}^{\mathbb{C}}(\infty)  &  :=\lim_{n\rightarrow\infty}\kappa
_{L}^{\mathbb{C}}(n)<\infty,\\
\kappa_{4/3}^{\mathbb{C}}(\infty)  &  :=\lim_{n\rightarrow\infty}\kappa
_{4/3}^{\mathbb{C}}(n)<\infty.
\end{align*}

The exact values of $\kappa_{O}^{\mathbb{C}}(n)$ and $\kappa_{L}^{\mathbb{C}%
}(n)$ are stated as an open problem in \cite[Page 31]{blei}. We solve this
problem here for $n=2$, with the aid of techniques introduced by Jameson,
\cite{jamilson}, concerning unital bilinear forms when dealing with a specific
form of two-dimensional Grothendieck's inequality. We will ultimately prove:

\begin{theorem}
\label{zqq}$\kappa_{O}^{\mathbb{C}}(2)=\kappa_{L}^{\mathbb{C}}(2)=\kappa
_{4/3}^{\mathbb{C}}(2)=1.$
\end{theorem}

\begin{proof}
Let $A,B$ be complex $C^{\ast}$-algebras with identities $e_{A},e_{B}$.
According to \cite{jamilson} we say that a bilinear form $V:$ $A\times
B\rightarrow\mathbb{C}$ is unital if
\[
V(e_{A},e_{B})=\Vert V\Vert=1.
\]
Note that if $A,B$ are finite-dimensional spaces and $T$ is any bilinear form
with $\Vert T\Vert=1$, then there will be unitary elements $x_{0}\in
A,y_{0}\in B$ such that $T(x_{0},y_{0})=1$, and then a unital form $V$ is
obtained by defining
\begin{equation}
V(x,y)=T(x_{0}x,y_{0}y). \label{mmm}%
\end{equation}
In fact, we have%
\[
V(e_{A},e_{B})=T(x_{0}e_{A},y_{0}e_{B})=T(x_{0},y_{0})=1,
\]%
\[
\left\Vert V(x,y)\right\Vert =\left\Vert T(x_{0}x,y_{0}y)\right\Vert
\leq\left\Vert T\right\Vert \left\Vert x_{0}x\right\Vert \left\Vert
y_{0}y\right\Vert \leq\left\Vert T\right\Vert \left\Vert x\right\Vert
\left\Vert y\right\Vert
\]
and thus $V(e_{A},e_{B})=\Vert V\Vert=1.$

Recall that $\mathbb{C}^{2}$ is a $C^{\ast}$-algebra with product $xy=\left(
x_{1}x_{2},y_{1}y_{2}\right)  $ and unit $e=e_{1}+e_{2}.$ Let $T\colon
\mathbb{C}^{2}\times\mathbb{C}^{2}\rightarrow\mathbb{C}$ be a bilinear form
with $\Vert T\Vert=1$. Then, by the Krein--Milman theorem there are extreme
elements of the closed unit ball of $\ell_{\infty}^{2},$ denoted by
$x_{0}=(\alpha_{1},\alpha_{2})$ and $y_{0}=(\beta_{1},\beta_{2})\in
\ell_{\infty}^{2}$ such that
\[
T(x_{0},y_{0})=\Vert T\Vert=1.
\]
It is well known that the extrema elements of the closed unit ball of
$\mathbb{C}^{2}$ have all coordinates with modulo $1$, see for instance
\cite[page 384]{diestel}. Hence $|\alpha_{i}|=|\beta_{j}|=1$, for all
$i,j\in\{1,2\}.$ Let is define the unital bilinear form $V$
\[
V(x,y)=T(x_{0}x,y_{0}y).
\]
One notes that
\begin{equation}
\sum\limits_{i=1}^{2}\left(  \sum\limits_{j=1}^{2}\left\vert V(e_{i}%
,e_{j})\right\vert ^{2}\right)  ^{1/2}=\sum\limits_{i=1}^{2}\left(
\sum\limits_{j=1}^{2}\left\vert T(e_{i},e_{j})\right\vert ^{2}\right)  ^{1/2};
\label{qqq}%
\end{equation}
indded
\begin{align*}
\sum\limits_{i=1}^{2}\left(  \sum\limits_{j=1}^{2}\left\vert V(e_{i}%
,e_{j})\right\vert ^{2}\right)  ^{1/2}  &  =\sum\limits_{i=1}^{2}\left(
\sum\limits_{j=1}^{2}\left\vert T(x_{0}e_{i},y_{0}e_{j})\right\vert
^{2}\right)  ^{1/2}\\
&  =\sum\limits_{i=1}^{2}\left(  \sum\limits_{j=1}^{2}\left\vert \alpha
_{i}\beta_{j}T(e_{i},e_{j})\right\vert ^{2}\right)  ^{1/2}=\sum\limits_{i=1}%
^{2}\left(  \sum\limits_{j=1}^{2}\left\vert T(e_{i},e_{j})\right\vert
^{2}\right)  ^{1/2}.
\end{align*}
Equality (\ref{qqq}), combined with the previous arguments, yields
\[
\sum\limits_{i=1}^{2}\left(  \sum\limits_{j=1}^{2}\left\vert T(e_{i}%
,e_{j})\right\vert ^{2}\right)  ^{1/2}\leq C\left\Vert T\right\Vert ,
\]
for all bilinear forms $T\colon\mathbb{C}^{2}\times\mathbb{C}^{2}%
\rightarrow\mathbb{C}$ with $\Vert T\Vert=1$ if, and only if,%
\[
\sum\limits_{i=1}^{2}\left(  \sum\limits_{j=1}^{2}\left\vert V(e_{i}%
,e_{j})\right\vert ^{2}\right)  ^{1/2}\leq C\left\Vert V\right\Vert
\]
for all unital bilinear forms $V\colon\mathbb{C}^{2}\times\mathbb{C}%
^{2}\rightarrow\mathbb{C}$ given by the method (\ref{mmm}). In conclusion, as
to understand the sharp constant problem -- objective of current study -- it
suffices to restrict the analysis to unital bilinear forms. Next we recall two
important pieces of information, namely \cite[Lemma 2.3]{jamilson} and
\cite[Theorem 1]{jamilson}, listed below for the readers' convenience:

(1) Any unital bilinear form $T\colon\mathbb{C}^{2}\times\mathbb{C}%
^{2}\rightarrow\mathbb{C}$ is of the form
\[
T(x,y)=\left(  a+ih\right)  x_{1}y_{1}+\left(  b-ih\right)  x_{1}%
y_{2}+(c-hi)x_{2}y_{1}+(d+hi)x_{2}y_{2},
\]
where each of $a+b,c+d,a+c,b+d,a+d,b+c$ is non-negative and $a+b+c+d=1$;

(2) \label{caracunit} Let $T\colon\mathbb{C}^{2}\times\mathbb{C}%
^{2}\rightarrow\mathbb{C}$ be a bilinear form given by%
\[
T(x,y)=\left(  a+ih\right)  x_{1}y_{1}+\left(  b-ih\right)  x_{1}%
y_{2}+(c-hi)x_{2}y_{1}+(d+hi)x_{2}y_{2}.
\]
Then $T$ is unital if and only if the following conditions hold:

(i) $a+b+c+d=1$;

(ii) each of $a+b,c+d,a+c,b+d,a+d,b+c$ is non-negative;

(iii) $h^{2}\leq bcd+acd+abd+abc$.

The above results allow us to re-state Blei's problem of finding
$L_{2}^{\mathbb{C}}(2)$ as an optimization problem:

Maximize the function $f\colon\mathbb{R}^{4}\rightarrow\lbrack0,\infty)$ given
by
\[
f(a,b,c,d,h)=\left(  a^{2}+b^{2}+2h^{2}\right)  ^{1/2}+\left(  c^{2}%
+d^{2}+2h^{2}\right)  ^{1/2}%
\]
when subject to the constrains
\[%
\begin{array}
[c]{ccccc}%
g_{1}(a,b,c,d,h) & = & -a-b & \leq & 0,\\
g_{2}(a,b,c,d,h) & = & -c-d & \leq & 0,\\
g_{3}(a,b,c,d,h) & = & -a-c & \leq & 0,\\
g_{4}(a,b,c,d,h) & = & -b-d & \leq & 0,\\
g_{5}(a,b,c,d,h) & = & -a-d & \leq & 0,\\
g_{6}(a,b,c,d,h) & = & -b-c & \leq & 0,\\
g_{7}(a,b,c,d,h) & = & h^{2}-(bcd+acd+abd+abc) & \leq & 0,\\
t(a,b,c,d,h) & = & a+b+c+d-1 & = & 0.
\end{array}
\]
Applying Karush--Kuhn--Tucker Theorem one finds the maximum of $f$ over that
set is precisely $1$, and hence we have proven
\begin{equation}
L_{2}^{\mathbb{C}}(2)=1. \label{vv}%
\end{equation}
Since
\[
\left(  \sum\limits_{j=1}^{n}\left(  \sum\limits_{i=1}^{n}\left\vert
T(e_{i},e_{j})\right\vert \right)  ^{2}\right)  ^{1/2}\leq\sum\limits_{i=1}%
^{2}\left(  \sum\limits_{j=1}^{2}\left\vert T(e_{i},e_{j})\right\vert
^{2}\right)  ^{1/2},
\]
by (\ref{vv}) and symmetry we have
\begin{equation}
1\leq\kappa_{O}^{\mathbb{C}}(2)\leq\kappa_{L}^{\mathbb{C}}(2)\leq1.
\label{111}%
\end{equation}
The H\"{o}lder inequality combined with (\ref{111}) gives us
\[
\kappa_{4/3}^{\mathbb{C}}(2)=1,
\]
which finally concludes the proof of Theorem \ref{zqq}.
\end{proof}

\subsection{Grothendieck's constants}

Let $K_{G}^{(m)}(d)$ be the optimal constant such that
\[
\left\vert
{\textstyle\sum\limits_{i=1}^{m}}
{\textstyle\sum\limits_{j=1}^{m}}
a_{ij}\left\langle x_{i},y_{j}\right\rangle \right\vert \leq K_{G}%
^{(m)}(d)\max\left\{  \left\vert
{\textstyle\sum\limits_{i=1}^{m}}
{\textstyle\sum\limits_{j=1}^{m}}
a_{ij}s_{i}t_{j}\right\vert :\left\vert s_{i}\right\vert \leq1,\left\vert
t_{j}\right\vert \leq1\right\}
\]
for all $d$-dimensional real Hilbert spaces $H$, all unit vectors
$x_{i},...,x_{m},y_{1},...,y_{m}\in H$ and all $m\times m$ scalar matrices
$a_{ij}$. Denoting%
\[
K_{G}(d):=\sup_{m}K_{G}^{(m)}(d),
\]
Grothendieck's theorem asserts that%
\[
K_{G}:=\sup_{d}K_{G}(d)<\infty.
\]
{For a detailed survey on the Grothendieck theorem we refer to \cite{pisier}}.
The constants $K_{G}$, $K_{G}(d)$ and $K_{G}^{(m)}(d)$ are, in general,
unknown (see, for instance, \cite{finch}) and important in physical problems
(see \cite{hua} and the references therein).

The problem of finding truncated sharp constants can be re-written as%
\begin{equation}
\underset{\{x_{i}\}_{i=1}^{m},\{y_{j}\}_{j=1}^{m}\subset\mathbb{S}^{d-1}}%
{\max}\left\vert \sum_{i=1}^{m}\sum_{j=1}^{m}a_{ij}\langle x_{i},y_{j}%
\rangle\right\vert \leq K_{G}^{(m)}(d)\underset{\left\vert s_{i}\right\vert
\leq1,\left\vert t_{j}\right\vert \leq1}{\max}\left\vert \sum_{i=1}^{m}%
\sum_{j=1}^{m}a_{ij}t_{i}s_{j}\right\vert , \label{7000}%
\end{equation}
where $\mathbb{S}^{d-1}=\{x\in\mathbb{R}^{d}:\sum_{i=1}^{d}x_{i}^{2}=1\}$
where $d$ is the dimension of the Hilbert space.

Another way to interpret (\ref{7000}) is by saying that for any positive
integers $m,d$ and any bilinear form $T:\mathbb{R}^{m}\times\mathbb{R}%
^{m}\rightarrow\mathbb{R}$ there holds
\begin{equation}
\underset{\{x_{i}\}_{i=1}^{m},\{y_{j}\}_{j=1}^{m}\subset\mathbb{S}^{d-1}}%
{\max}\left\vert \sum_{i=1}^{m}\sum_{j=1}^{m}T(e_{i},e_{j})\langle x_{i}%
,y_{j}\rangle\right\vert \leq K_{G}^{(m)}(d)\Vert T\Vert, \label{0099}%
\end{equation}
as
\[
\Vert T\Vert=\underset{\left\vert s_{i}\right\vert \leq1,\left\vert
t_{j}\right\vert \leq1}{\max}\left\vert T\left(  \sum_{i=1}^{m}t_{i}e_{i}%
,\sum_{j=1}^{m}s_{j}e_{j}\right)  \right\vert =\underset{\left\vert
s_{i}\right\vert \leq1,\left\vert t_{j}\right\vert \leq1}{\max}\left\vert
\sum_{i=1}^{m}\sum_{j=1}^{m}T\left(  e_{i},e_{j}\right)  t_{i}s_{j}\right\vert
.
\]
By (\ref{0099}) it is obvious that
\[
K_{G}^{(m)}(d)=\underset{\Vert T\Vert\leq1}{\sup}\left(  \underset
{\{x_{i}\}_{i=1}^{m},\{y_{j}\}_{j=1}^{m}\subset\mathbb{S}^{d-1}}{\max
}\left\vert \sum_{i=1}^{m}\sum_{j=1}^{m}T(e_{i},e_{j})\langle x_{i}%
,y_{j}\rangle\right\vert \right)  .
\]
Thus, finding the sharp values of $K_{G}^{(m)}(d)$ is equivalent to finding
the maximum of the function%
\begin{align*}
f_{m,d}  &  :B_{\mathcal{L}(^{2}\mathbb{R}^{m};\mathbb{R})}\rightarrow
\mathbb{R}\\
f_{m,d}(T)  &  =\underset{\{x_{i}\}_{i=1}^{m},\{y_{j}\}_{j=1}^{m}%
\subset\mathbb{S}^{d-1}}{\max}\left\vert \sum_{i=1}^{m}\sum_{j=1}^{m}%
T(e_{i},e_{j})\langle x_{i},y_{j}\rangle\right\vert
\end{align*}
where $B_{\mathcal{L}(^{2}\mathbb{R}^{m};\mathbb{R})}$ denotes the closed unit
ball of the space of bilinear forms $T\colon\mathbb{R}^{m}\times\mathbb{R}%
^{m}\rightarrow\mathbb{R}$.

The following lemma is straightforward:

\begin{lemma}
Let $m,d$ be positive integers. The function $f_{m,d}$ is continuous and convex.
\end{lemma}

Since $f_{m,n}$ is continuous and convex and $B_{\mathcal{L}(^{2}%
\mathbb{R}^{m},\mathbb{R}^{n};\mathbb{R})}$ is convex and compact {we have the
following result}:

\begin{proposition}
For all positive integers $m,d$ we have%
\[
K_{G}^{(m)}(d)=\max_{T\in\mathcal{C}}\left(  \underset{\{x_{i}\}_{i=1}%
^{m},\{y_{j}\}_{j=1}^{m}\subset\mathbb{S}^{d-1}}{\max}\sum_{i=1}^{m}\sum
_{j=1}^{m}T(e_{i},e_{j})\langle x_{i},y_{j}\rangle\right)
\]
and%
\[
K_{G}(d)=\sup_{m}\left(  \max_{T\in\mathcal{C}}\underset{\{x_{i}\}_{i=1}%
^{m},\{y_{j}\}_{j=1}^{m}\subset\mathbb{S}^{d-1}}{\max}\sum_{i=1}^{m}\sum
_{j=1}^{m}T(e_{i},e_{j})\langle x_{i},y_{j}\rangle\right)  .
\]

\end{proposition}

Since $\mathcal{C}$ is finite and fully determined, the task reduces to
calculate%
\[
\underset{\{x_{i}\}_{i=1}^{m},\{y_{j}\}_{j=1}^{m}\subset\mathbb{S}^{d-1}}%
{\max}\sum_{i=1}^{m}\sum_{j=1}^{m}T_{0}(e_{i},e_{j})\langle x_{i},y_{j}\rangle
\]
for all $T_{0}\in\mathcal{C}$ and this can be easily calculated by the
Lagrange Multipliers method.  

\bigskip

\end{document}